\documentclass[10pt]{article}

\usepackage{indentfirst}
\usepackage[a4paper,margin=1in]{geometry}
\usepackage{fontspec}
\usepackage{microtype}
\usepackage{amsmath,amssymb,amsthm,mathtools}
\usepackage{booktabs,array}
\usepackage{hyperref}
\hypersetup{colorlinks=true,linkcolor=blue,citecolor=blue,urlcolor=blue}
\usepackage{enumitem}
\setlist{nosep}
\allowdisplaybreaks

\newtheorem{theorem}{Theorem}[section]
\newtheorem{lemma}{Lemma}[section]
\newtheorem{proposition}{Proposition}[section]
\newtheorem{corollary}{Corollary}[section]
\newtheorem{conjecture}{Conjecture}[section]
\theoremstyle{remark}
\newtheorem{remark}{Remark}[section]

\newcommand{\TT}{\mathrm{TT}}
\newcommand{\Reg}{\mathrm{RT}_5}
\newcommand{\ex}{\operatorname{ex}}
\newcommand{\ind}{\mathbf{1}}
\newcommand{\dd}{\overleftrightarrow{K}}
\newcommand{\dT}{\overleftrightarrow{T}}

\title{Degree-square Tur\'an problem for two self-converse tournaments}
\author{Zhuoran HAN, Yaojun CHEN\footnote{Corresponding author. Email: yaojunc@nju.edu.cn}
\\
{\small School of Mathematics, Nanjing University, Nanjing 210093, China}
 }

\date{}

\begin{document}
\maketitle

\begin{abstract}

For a fixed digraph $F$, let $\operatorname{ex}_2^+(n,F)$ be the maximum of $\sum_{v\in V(D)}d_D^+(v)^2$ over all $n$-vertex $F$-free digraphs. Ai et al. [arXiv:2606.03520, 2026] asked for which self-converse tournament $F$ one can determine $\operatorname{ex}_2^+(n,F)$. Let $\TT_r$ denote a transitive tournament on $r$ vertices and $\text{RT}_5$ denote a regular tournament on 5 vertices. Note that both $\TT_r$ and $\text{RT}_5$ are self-converse. I\v{l}kovi\v{c} [arXiv:2609.05042, 2026] determined $\operatorname{ex}_2^+(n,F)$ for $F=\TT_4$. In this paper, we determine $\operatorname{ex}_2^+(n,F)$ for $F$ being $\TT_r$ or $\text{RT}_5$, and the latter confirms a conjecture due to
I\v{l}kovi\v{c}.


\end{abstract}

\medskip

\section{Introduction}
For an undirected graph $H$, the classical Tur\'an number $\ex(n,H)$ is the maximum
number of edges in an $n$-vertex $H$-free graph. Turán theory dates back to the early 1940’s when Turán \cite{Turan} determined $\ex(n, K_r)$ and the corresponding extremal graph for the complete graph $K_r$.

\begin{theorem}[Tur\'an \cite{Turan}]
For every $r\ge2$,  \[\ex(n,K_r)=e\!\left(T_{r-1}(n)\right),\] where $T_{r-1}(n)$ is the balanced complete $(r-1)$-partite graph on $n$ vertices. Moreover, $T_{r-1}(n)$ is the unique extremal $K_r$-free graph up to isomorphism.
\end{theorem}

For a graph $G$ and a vertex $v$ of $G$, let $d_G(v)$ be the degree of $v$ in $G$. It is well known that 
the size $e(G)=\frac{1}{2}\sum_{v\in V(G)}d_G(v)$. Thus, the ordinary Tur\'an problem can be viewed as the maximization of the sum of all vertex degrees. 
This naturally leads to the more general problem of maximizing
\[
\sum_{v\in V(G)}\varphi\!\left(d_G(v)\right)
\]
over all $H$-free graphs, where $\varphi$ is a nondecreasing function.

Caro and Yuster~\cite{CaroYuster} initiated the systematic study of the
power-sum version. For a positive integer $p$, they defined
\[
e_p(G)=\sum_{v\in V(G)}d_G(v)^p,
\qquad
\ex_p(n,H)
=
\max\left\{
e_p(G):
|V(G)|=n,\;
G\text{ is }H\text{-free}
\right\}.
\]
For forbidden complete graphs, they proved that the balanced Tur\'an graph
is extremal for $p=2$ and $p=3$, while it need not remain extremal for
$p\ge4$. They also determined $\ex_p(n,P_\ell)$ for every fixed path
$P_\ell$, every $p\ge2$, and all sufficiently large $n$. Bollob\'as and Nikiforov~\cite{BollobasNikiforov}
substantially sharpened the clique case.
For every fixed integer $r\ge2$ and real $p$ with $0<p<r$,
they proved that the balanced complete $r$-partite graph maximizes
the degree $p$-power sum among $n$-vertex $K_{r+1}$-free graphs
for all sufficiently large $n$.
In contrast, for $p\ge r+\lceil\sqrt{2r}\rceil$, this graph
is not extremal when $n$ is sufficiently large.

Analogous studies can be extended to digraphs. Throughout, a digraph is finite and loopless, with no parallel arcs in the same direction; opposite arcs are allowed and are counted separately. 

We write \(x\leftrightarrow y\) if both \(xy\) and \(yx\) are arcs, such a pair is called a \emph{digon}. We write \(\dd\!_m\) for the complete bidirected digraph on \(m\) vertices. More generally, Let $\dd\!_{n_1,\ldots,n_t}$ denote the complete bidirected multipartite digraph with independent parts of sizes $n_1,\ldots,n_t$, in which both opposite arcs are present between every two vertices in distinct parts. Let $\dT\!_p(n)$ denote the balanced complete bidirected $p$-partite digraph, i.e., the complete bidirected $p$-partite digraph whose $p$ part sizes differ by at most one; empty parts are allowed when $n<p$.

A digraph $D$ is $F$-free if it contains no copy of $F$ as a subdigraph. For a function
$\varphi\colon\mathbb Z_{\ge0}\to\mathbb R$, write
\[
\Phi_\varphi(D)
=
\sum_{v\in V(D)}
\varphi\!\left(d_D^+(v)\right),
\qquad
\ex_\varphi^+(n,F)
=
\max\left\{
\Phi_\varphi(D):
|V(D)|=n,\;
D\text{ is }F\text{-free}
\right\}.
\]

In the square case $\varphi(t)=t^2$, we abbreviate these quantities to
$\Phi(D)$ and $\ex_2^+(n,F)$. We also write $e(D)=|A(D)|$ and use
$\ex(n,F)$ for the ordinary maximum number of arcs in an $n$-vertex
$F$-free digraph. Since
\[
e(D)=\sum_{v\in V(D)}d_D^+(v),
\]
the ordinary Tur\'an problem for digraphs is the linear case, while
$\ex_2^+(n,F)$ is its quadratic refinement.
Ordinary Tur\'an problems for digraphs go back to Brown and
Harary~\cite{BrownHarary}. In particular, they proved the following theorem.
\begin{theorem}[Brown and
	Harary \cite{BrownHarary}]\label{thm:Brown-Harary}
    For every tournament $F$ on $r\ge2$ vertices,
\[
	\ex(n,F)=2\ex(n,K_r),
\]
where $\ex(n,K_r)$ denotes the ordinary undirected Tur\'an number.
\end{theorem}
Thus the ordinary arc extremum for
a forbidden tournament depends only on its order, whereas nonlinear degree
functionals can distinguish its orientation.

The systematic study of the degree-square refinement for digraphs was
initiated only very recently by Lin, Wang, Zhou and
Zhou~\cite{LinWangZhouZhou}. They defined the degree squared sum of a
digraph and posed the general problem of maximizing it under a
forbidden-digraph condition. As a principal example, writing $\TT_r$ for
the transitive tournament on $r$ vertices, they proved that every
$n$-vertex $\TT_r$-free digraph satisfies
\begin{equation*}
	\max\left\{
	\sum_{v\in V(D)}d_D^+(v)^2,\,
	\sum_{v\in V(D)}d_D^-(v)^2
	\right\}
	\le
	\left(\frac{r-2}{r-1}\right)^2n^3.
\end{equation*}
Their proof uses the out-neighborhood of a maximum-out-degree vertex and
induction on $r$.

Ai et al.~\cite{AiEtAl} subsequently formulated the
one-sided degree-square Tur\'an number $\ex_2^+(n,F)$ in connection with
finite palette extremality. They developed a prefix-majorization method for
nondecreasing convex out-degree functionals, determined the corresponding
functionals for directed cycles, and obtained the exact result for $\TT_3$. They asked for which self-converse tournaments $F$ the degree-square Tur\'an number $\ex_2^+(n,F)$ can be determined exactly. I\v{l}kovi\v{c}~\cite{Ilkovic} then determined the
exact degree-square Tur\'an numbers for $\TT_4$ and the strongly
connected four-vertex tournament $R_4$. Let $\Reg$ be the unique regular tournament on five vertices. For convenience, we construct $\Reg$ as follows: the vertex set is $\mathbb Z_5$, and
\[
  i\to j
  \quad\Longleftrightarrow\quad
  j-i\in\{1,2\}\pmod 5.
\]
\begin{conjecture}[I\v{l}kovi\v{c}~\cite{Ilkovic}]\label{C1} For every $n$, $\ex_2^+(n,\Reg)=\Phi(\dT\!_4(n))$.
\end{conjecture}

Note that both $\TT_r$ and $\Reg$ are
self-converse. In this paper, we study the degree-square Tur\'an problem for $\TT_r$ and $\Reg$.

We first establish a pointwise out-degree domination theorem for $\TT_r$-free digraphs.

\begin{theorem}\label{thm:pointwise}
Let $r\ge2$ and $m\ge1$, and let $D$ be an $m$-vertex $\TT_r$-free digraph. Put $k=\min\{m,r-1\}$. There exists a complete bidirected $k$-partite digraph $H$, with $k$ nonempty parts and on the same vertex set as $D$, such that
\[
  d_D^+(v)\le d_H^+(v)\qquad\text{for every }v\in V(D).
\]
Moreover,
\[
  d_D^+(v)=d_H^+(v)\text{ for every }v\in V(D)
  \quad\Longrightarrow\quad D=H.
\]
\end{theorem}

\begin{corollary}\label{cor:partition}
Let $\varphi\colon\mathbb Z_{\ge0}\to\mathbb R$ be nondecreasing, let $r\ge2$ and $m\ge1$, and put $k=\min\{m,r-1\}$. Then
\[
  \ex_\varphi^+(m,\TT_r)
  =\max_{\substack{m_1+\cdots+m_k=m\\m_i\ge1}}
    \sum_{i=1}^k m_i\varphi(m-m_i).
\]
If $\varphi$ is strictly increasing, an $m$-vertex $\TT_r$-free digraph is extremal if and only if it is a complete bidirected $k$-partite digraph whose part sizes attain this maximum.
\end{corollary}

Together with Erd\H{o}s' undirected domination theorem~\cite{Erdos}, this gives
\begin{equation*}
  \ex_\varphi^+(m,\TT_r)
  =\max_{\substack{|V(G)|=m\\K_r\nsubseteq G}}
    \sum_{v\in V(G)}\varphi\!\left(d_G(v)\right)
\end{equation*}
for every nondecreasing $\varphi$, where $G$ ranges over simple undirected graphs.

 For the square function, the partition problem has a unique balanced solution.

\begin{theorem}\label{thm:transitive-square}
Let $r\ge2$ and $m\ge1$, put $p=r-1$, and write $m=pq+s$ with $0\le s<p$. Then
\[
  \ex_2^+(m,\TT_r)
  =s(q+1)(m-q-1)^2+(p-s)q(m-q)^2.
\]
An $m$-vertex $\TT_r$-free digraph attains this value if and only if it is isomorphic to $\dT\!_p(m)$.
\end{theorem}

 %I\v{l}kovi\v{c}~\cite{Ilkovic} conjectured that $T_4(n)$ attains $\ex_2^+(n,\Reg)$ for every $n$. 
 The next theorem confirms Conjecture \ref{C1} and characterizes all extremal digraphs.

\begin{theorem}\label{thm:reg5-main}
For every integer $n\ge1$, write $n=4q+\rho$ with $0\le\rho<4$. Then
\[
  \ex_2^+(n,\Reg)
  =\rho(q+1)(3q+\rho-1)^2+(4-\rho)q(3q+\rho)^2.
\]
An $n$-vertex $\Reg$-free digraph attains this value if and only if it is isomorphic to $\dT\!_4(n)$.
\end{theorem}

This paper is organized as follows. Section~2 proves the results for transitive tournaments.
Section~3 collects the symmetrization framework, local restrictions, and some notation for the tournament $\Reg$.
Section~4 treats the case in which the symmetrized quotient contains no
bidirected $K_4$.
Section~5 handles the remaining case by means of a four-vertex deletion
argument and completes the proof of Theorem~\ref{thm:reg5-main}.
Finally, Section~6 gives some concluding remarks and possible directions
for further work.

\section{Transitive tournaments}\label{sec:transitive}

We use the following elementary balancing calculation.

\begin{lemma}\label{lem:balance}
Let $M,t,h$ be positive integers with $t\le M$. Among all $h$-tuples of nonnegative integers
$x_1,\ldots,x_h$ with $\sum_i x_i=t$, the quantity
\[
  \sum_{i=1}^h x_i(M-x_i)^2
\]
is maximized if and only if the $x_i$ differ by at most one.
\end{lemma}

\begin{proof}
Let $f(x)=x(M-x)^2$. Suppose $x_i\ge x_j+2$. Then
\[
  f(x_i-1)+f(x_j+1)-f(x_i)-f(x_j)
  =(x_i-x_j-1)(4M-3x_i-3x_j).
\]
The first factor is positive. Since $x_i+x_j\le t\le M$, the second factor is at least $M>0$.
Moving one unit from $x_i$ to $x_j$ therefore strictly increases the sum. Repeating this
operation proves the lemma.
\end{proof}

\begin{proof}[\bfseries{Proof of Theorem~\ref{thm:pointwise}}]
We prove the statement by induction on $r$. For $r=2$, a $\TT_2$-free digraph has no arcs.
More generally, if $m<r$, take $H=\dd\!_m$, viewed as a complete bidirected $m$-partite
digraph with singleton parts. Then
\[
  d_D^+(v)\le m-1=d_H^+(v)
\]
for every $v$, and equality for all vertices forces $D=H$.

Assume $r\ge3$ and $m\ge r$, and suppose the theorem holds for $\TT_{r-1}$. Choose a vertex
$w$ of maximum out-degree and write
\[
  \Delta=d_D^+(w),\qquad X=N_D^+(w),\qquad
  Y=V(D)\setminus X,\qquad c=|Y|=m-\Delta.
\]
Since $w\in Y$, we have $c\ge1$. The digraph $D[X]$ is $\TT_{r-1}$-free. By induction there
is a complete bidirected $t$-partite graph $H_X$ on $X$, where
$t=\min\{\Delta,r-2\}$, such that
\[
  d_{D[X]}^+(x)\le d_{H_X}^+(x)\qquad(x\in X),
\]
with the corresponding rigidity property. If $X=\varnothing$, take $H_X$ empty and $t=0$.

Form $H_0$ by taking $H_X$ on $X$, making $Y$ one independent part, and placing both
opposite arcs between every vertex of $X$ and every vertex of $Y$. For $x\in X$,
\[
  d_D^+(x)\le d_{D[X]}^+(x)+c
  \le d_{H_X}^+(x)+c=d_{H_0}^+(x),
\]
and for $y\in Y$, maximality of $d_D^+(w)$ gives
\[
  d_D^+(y)\le\Delta=|X|=d_{H_0}^+(y).
\]
Thus $H_0$ pointwise dominates $D$.

The graph $H_0$ has $t+1\le r-1$ nonempty parts. If $t+1<r-1$, then because $m\ge r$
we can repeatedly split a non-singleton part into two nonempty parts and insert all digons
between them until exactly $r-1$ nonempty parts are present. Every such split strictly increases
the out-degree of vertices in the split part and leaves all other out-degrees unchanged. The
resulting $H$ still dominates $D$, and at least one vertex is strictly dominated; hence the rigidity
implication is automatic.

It remains to consider $t+1=r-1$. Set $H=H_0$ and suppose
\[
  d_D^+(v)=d_H^+(v)\qquad\text{for every }v\in V(D).
\]
Equality must hold throughout the two inequalities for each $x\in X$. Thus
\[
  d_{D[X]}^+(x)=d_{H_X}^+(x)\qquad\text{for every }x\in X,
\]
every arc from $X$ to $Y$ is present, and inductive rigidity gives $D[X]=H_X$. Since
$t=r-2$, write the nonempty parts of $H_X$ as $A_1,\ldots,A_{r-2}$.

We claim that $Y$ is independent. If $yz$ were an arc of $D[Y]$, choose $x_i\in A_i$ for each
$i$. Between distinct $A_i$ both arcs are present, and all arcs from $X$ to $Y$ are present.
Hence the order
\[
  x_1,x_2,\ldots,x_{r-2},y,z
\]
contains a copy of $\TT_r$, a contradiction. Thus $Y$ is independent. Finally every $y\in Y$
satisfies
\[
  d_D^+(y)=d_H^+(y)=\Delta=|X|.
\]
Since $Y$ is independent, $y$ sends an arc to every vertex of $X$. Hence every pair across
$X,Y$ is a digon and $D=H$.
\end{proof}

\begin{proof}[\bfseries{Proof of Corollary~\ref{cor:partition}}]
Let $D$ be any $m$-vertex $\TT_r$-free digraph and let $H$ be supplied by
Theorem~\ref{thm:pointwise}. Since $\varphi$ is nondecreasing,
\[
  \Phi_\varphi(D)\le\Phi_\varphi(H).
\]
If the $k$ nonempty parts of $H$ have sizes $m_1,\ldots,m_k$, then every vertex in part $i$
has out-degree $m-m_i$, so
\[
  \Phi_\varphi(H)=\sum_{i=1}^k m_i\varphi(m-m_i).
\]
Conversely, every complete bidirected $k$-partite digraph is $\TT_r$-free, giving the reverse
inequality.

If $\varphi$ is strictly increasing and $D$ is extremal, then $H$ is also $\TT_r$-free, so extremality
and pointwise domination force
\[
  d_D^+(v)=d_H^+(v)\qquad\text{for every }v.
\]
The rigidity clause of Theorem~\ref{thm:pointwise} yields $D=H$. The converse is immediate.
\end{proof}

\begin{remark}
The strict monotonicity assumption cannot be omitted in general from the equality classification. For example, if $\varphi$ is constant,
every $\TT_r$-free digraph is extremal. Thus nondecreasing functions admit a multipartite
extremizer, whereas strictly increasing functions force every extremizer to be multipartite.
\end{remark}

\begin{proof}[\bfseries{Proof of Theorem~\ref{thm:transitive-square}}]
Put $p=r-1$. By Corollary~\ref{cor:partition},
\[
  \ex_2^+(m,\TT_r)=\max\sum_{i=1}^p m_i(m-m_i)^2,
\]
where zero parts may be included when $m<p$. Lemma~\ref{lem:balance}, with $M=t=m$ and
$h=p$, shows that the unique maximizing multiset of part sizes is the balanced $p$-partition.
Thus if $m=pq+s$, exactly $s$ parts have size $q+1$ and the remaining $p-s$ parts have size
$q$. Substitution gives the stated formula. Since $t^2$ is strictly increasing on the nonnegative
integers, Corollary~\ref{cor:partition} gives the equality classification.
\end{proof}

\section{Preliminaries for the tournament \texorpdfstring{$\text{RT}_5$}{RT5}}\label{sec:reg5-prelim}

A digraph is \emph{semicomplete} if every two distinct vertices are joined by at least one arc. Let $Q$ be a semicomplete digraph. Replace each vertex $u\in V(Q)$ by a nonempty independent set $V_u$, and for each quotient arc $u\to v$ insert all arcs from $V_u$ to $V_v$. The resulting digraph is called an \emph{independent blow-up} of $Q$. Class sizes always refer to the numbers of actual vertices. If $F$ is a tournament, then an independent blow-up contains $F$ if and only if its quotient contains $F$. An induced subdigraph of an independent blow-up is again an independent blow-up of the corresponding induced quotient after empty classes are deleted.

For $X\subseteq V(D)$, let
\[
  e_D(X)=|A(D[X])|,
\]
and omit the subscript when the ambient digraph is clear. For disjoint sets $X,Z$, let $e_D(X,Z)$ be the number of arcs directed from $X$ to $Z$, and put
\[
  d_X^+(v)=|N_D^+(v)\cap X|.
\]

We use the clone-symmetrization argument of I\v{l}kovi\v{c}~\cite{Ilkovic}. The following formulation records the labelled out-degrees needed to recover the original extremal digraph.

\begin{lemma}\label{lem:reg5-symmetrization}
Let $F$ be a tournament, and let $D$ maximize $\Phi$ among all
$F$-free digraphs of its order. Define a relation $\sim$ on $V(D)$ by
\[
  x\sim y
  \quad\Longleftrightarrow\quad
  x=y\ \text{or $x$ and $y$ are nonadjacent}.
\]
Then $\sim$ is an equivalence relation whose equivalence classes
$V_1,\ldots,V_s$ are independent. Choose $v_i\in V_i$, let
\[
  Q=D[\{v_1,\ldots,v_s\}],
\]
and let $\widehat D$ be the independent blow-up of $Q$ on the
classes $V_i$. Then $Q$ is semicomplete and $F$-free, $\widehat D$ is $F$-free, and
\[
  d_{\widehat D}^+(x)=d_D^+(x)
  \qquad(x\in V(D)).
\]
\end{lemma}

\begin{proof}
For nonadjacent vertices $x,y$, let $D_{x\to y}$ be obtained from
$D$ by making $x$ a nonadjacent clone of $y$: for every
$z\in V(D)\setminus\{x,y\}$,
\[
  xz\in A(D_{x\to y})
  \quad\Longleftrightarrow\quad
  yz\in A(D),
\]
and
\[
  zx\in A(D_{x\to y})
  \quad\Longleftrightarrow\quad
  zy\in A(D),
\]
while $x$ and $y$ remain nonadjacent. All other arcs are unchanged.
This preserves $F$-freeness: a tournament copy uses at most one of $x,y$, and a copy using the new $x$ can be transferred to the original $y$. Put
\[
  \varepsilon_z=\ind_{zy\in A(D)}-\ind_{zx\in A(D)}
  \qquad(z\ne x,y).
\]
The two-direction cloning calculation gives
\[
  0\ge\Phi(D_{x\to y})+\Phi(D_{y\to x})-2\Phi(D)
   =2\sum_{z\ne x,y}\varepsilon_z^2\ge0.
\]
Thus all $\varepsilon_z$ vanish. The individual changes then become
$d_D^+(y)^2-d_D^+(x)^2$ and its negative; both are nonpositive. Hence
\[
  N_D^-(x)=N_D^-(y),\qquad d_D^+(x)=d_D^+(y).
\]
If $x,y$ and $y,z$ are nonadjacent, an arc $x\to z$ would force $x\to y$, and an arc $z\to x$ would force $z\to y$. This proves the transitivity of $\sim$, and hence $\sim$ is an equivalence relation.

The resulting classes are independent and distinct classes are adjacent, so $Q$ is semicomplete; it is $F$-free as an induced subdigraph of $D$. All vertices of a target class $V_j$ have the same in-neighborhood. Therefore, for $x\in V_i$,
\[
  d_{\widehat D}^+(x)
  =\sum_{j:\,v_i\to v_j}|V_j|
  =d_D^+(v_i)=d_D^+(x).
\]
The projection property for tournament copies shows that $\widehat D$ is $F$-free.
\end{proof}

Let $R_4$ be the tournament on vertex set $\{1,2,3,4\}$ with arc set
\[
  \{12,23,34,41,13,24\},
\]
where $ij$ denotes the arc $i\to j$. It is clear that,
\[
  \Reg-v\cong R_4
  \qquad(v\in V(\Reg)).
\]
In a semicomplete digraph, any four vertices containing a directed $4$-cycle contain a copy of $R_4$: retain the four cycle arcs and choose one available direction on each diagonal.

For a semicomplete digraph $Q$, let $M(Q)$ be its undirected \emph{digon graph}: $uv\in E(M(Q))$ if and only if $u\leftrightarrow v$ in $Q$.

\begin{lemma}\label{lem:reg5-local}
If $Q$ is a $\Reg$-free semicomplete digraph, then $M(Q)$ is $K_{2,3}$-free. Moreover, for every vertex $v$, the quotient
\[
  Q[N_{M(Q)}(v)]
\]
is $R_4$-free.
\end{lemma}

\begin{proof}
In $\Reg$, the vertices $\{0,1,2\}$ induce a transitive triangle, while $\{0,1,3\}$ induce a cyclic triangle; in each case the other two vertices are joined by an arc. Thus $\Reg$ can be partitioned either into a transitive triangle and an arc or into a cyclic triangle and an arc.

If $M(Q)$ contained a $K_{2,3}$, choose one available direction on every pair on the three-vertex side. The resulting tournament is either transitive or cyclic. Choose an arc on the two-vertex side. Every cross-pair is a digon, so all cross-arcs required by the corresponding decomposition of $\Reg$ can be selected, a contradiction. In particular, any two quotient vertices have at most two common digon-neighbors.

If the digon-neighborhood of some vertex $v$ contained $R_4$, then the digons between $v$ and those four vertices would complete a copy of $\Reg$, again a contradiction.
\end{proof}

By Theorem~\ref{thm:Brown-Harary}, every $t$-vertex $R_4$-free digraph $H$ satisfies
\[
  e(H)\le2\ex(t,K_4)\le\frac23t^2.
\]
For a more restricted quotient, we need the following improvement.

\begin{lemma}\label{lem:local-arc-bound}
Let $H$ be a $t$-vertex independent blow-up whose semicomplete quotient $Q_H$ is $R_4$-free and contains no $\dd\!_3$. Then
\begin{equation*}
  e(H)\le\frac47t^2.
\end{equation*}
\end{lemma}

\begin{proof}
The case $t=0$ is immediate. The digon graph $M(Q_H)$ has neither a four-vertex path nor a $K_{1,3}$. Indeed, a digon path $a-b-c-d$, together with either available direction between its endpoints, gives a directed four-cycle. For a digon star centered at $v$, choose a directed path $a\to b\to c$ among the three leaves; then $v\to a\to b\to c\to v$ is such a cycle. Either cycle yields $R_4$.

Consequently, each component of $M(Q_H)$ is $K_1$, $K_2$ or $P_3$, since $K_3$ is excluded. A component of total blow-up size $u$ contributes at most $u^2/2$ arcs in the first two cases. For a $P_3$ component, write $a$ for the center-class size and $p,q$ for the end-class sizes, with $p+q=u-a$. Its internal arc count satisfies
\[
  2a(p+q)+pq
  \le2a(u-a)+\frac{(u-a)^2}{4}
  =\frac47u^2-\frac{(7a-3u)^2}{28}
  \le\frac47u^2.
\]
Between distinct digon components, every pair of quotient vertices has exactly one arc. If the component sizes are $u_1,\ldots,u_k$, then
\[
  e(H)\le\frac47\sum_i u_i^2+\sum_{i<j}u_i u_j
  =\frac12t^2+\frac1{14}\sum_i u_i^2
  \le\frac47t^2.
\]
\end{proof}

For a nonnegative integer $n=4q+\rho$, where $0\le\rho<4$, define
\[\begin{aligned}
  p_4(n)&=\rho(q+1)^2+(4-\rho)q^2,\\
  e_4(n)&=n^2-p_4(n),\\
  B_4(n)&=\rho(q+1)(n-q-1)^2+(4-\rho)q(n-q)^2.
\end{aligned}\]
These are, respectively, the sum of squares of the part sizes of $\dT\!_4(n)$, the number of arcs in $\dT\!_4(n)$, and $\Phi(\dT\!_4(n))$. We set all three quantities equal
to zero at $n=0$.

\section{Quotients without a bidirected \texorpdfstring{$K_4$}{K4}}\label{sec:no-bidirected-k4}

Throughout this section, $D$ is an $n$-vertex independent blow-up whose semicomplete quotient $Q$ is $\Reg$-free and contains no $\dd\!_4$. We will prove that $\Phi(D)<B_4(n)$ whenever $n\ge5$.

For a vertex $x$, write
\[
  \sigma_D(x)=d_D^+(x)+d_D^-(x).
\]

\begin{lemma}\label{lem:low-total-degree}
Under the assumptions of this section, if $n\ge1$, then some vertex $x$ satisfies
\begin{equation*}
  \sigma_D(x)<\frac{10}{7}n.
\end{equation*}
\end{lemma}

\begin{proof}
Suppose instead that every vertex has total degree at least $10n/7$. Let $V_u$ be a quotient class of size $a_u$, and let $\beta_u$ be the total size of all other classes bidirected to $V_u$. Semicompleteness supplies at least one arc between a vertex of $V_u$ and each of the $n-a_u$ vertices outside its class, while every digon contributes one additional incident arc. Hence
\[
  \sigma_D(x)=n-a_u+\beta_u
  \qquad(x\in V_u),
\]
and therefore
\begin{equation}\label{eq:beta-lower}
  \beta_u\ge a_u+\frac37n
\end{equation}
for every quotient vertex $u$.

Choose a largest class $V_v$ and put
\[
  a=|V_v|=\max_u a_u.
\]
Let $B$ be the union of the classes that are bidirected to $v$, and put
\[
  C=V(D)\setminus(V_v\cup B).
\]
Write
\[
  b=|B|=\beta_v,
  \qquad
  c=|C|=n-a-b.
\]
From \eqref{eq:beta-lower} and $b\le n-a$,
\begin{equation*}
  a+\frac37n\le b\le n-a,
  \qquad
  0<a\le\frac27n.
\end{equation*}

By Lemma~\ref{lem:reg5-local}, the quotient induced by $B$ is $R_4$-free. It also contains no bidirected $K_3$, since adding $v$ would produce a bidirected $K_4$. Lemma~\ref{lem:local-arc-bound} therefore gives
\begin{equation*}
  e(B)\le\frac47b^2.
\end{equation*}
Let $e_2(B)$ be the number of unordered pairs of actual vertices in $B$ that form digons. Every pair from distinct quotient classes contributes one arc, and a digon contributes one additional arc, so
\[
  e(B)
  =\frac12\left(b^2-\sum_{u:\,V_u\subseteq B}a_u^2\right)+e_2(B).
\]
Together with the preceding bound on $e(B)$, this yields
\begin{equation*}
  2e_2(B)
  \le\frac17b^2+\sum_{u:\,V_u\subseteq B}a_u^2.
\end{equation*}

Now count the total number of digon-neighbors of actual vertices in $B$:
\[
  I_B=\sum_{u:\,V_u\subseteq B}a_u\beta_u.
\]
By \eqref{eq:beta-lower},
\begin{equation*}
  I_B
  \ge\sum_{u:\,V_u\subseteq B}a_u^2+\frac37nb.
\end{equation*}
On the other hand, digons whose other endpoint is also in $B$ contribute $2e_2(B)$, and those whose other endpoint is in $V_v$ contribute $ab$. Fix a quotient vertex $w$ represented in $C$. Its digon-neighbor classes in $B$ are common digon-neighbors of $v$ and $w$. By Lemma~\ref{lem:reg5-local}, there are at most two such classes, each of size at most $a$. Thus every actual vertex in $C$ has at most $2a$ digon-neighbors in $B$, and
\begin{equation*}
  I_B\le2e_2(B)+ab+2ac.
\end{equation*}
Combining the preceding three estimates and cancelling the class-square sum gives
\begin{equation*}
  \frac37nb\le\frac17b^2+ab+2ac.
\end{equation*}
Substituting $c=n-a-b$ and multiplying by seven, we obtain
\begin{equation*}
  F(b):=b^2-(7a+3n)b+14a(n-a)\ge0.
\end{equation*}

However, $a+\frac37n\le b\le n-a$ and $0<a\le\frac27n$, while $F$ is strictly negative at both endpoints:
\begin{align*}
  F\left(a+\frac37n\right)
  &=-20\left(a-\frac{31}{140}n\right)^2-\frac{17}{140}n^2<0,\\
  F(n-a)
  &=(n-a)(6a-2n)
   \le-\frac27n(n-a)<0.
\end{align*}
The second inequality uses $a\le2n/7$. It is clear that $F(b)<0$, contradicting $F(b)\ge0$.
\end{proof}

\begin{lemma}\label{lem:finite-blowup-arc}
Under the assumptions of this section,
\begin{equation*}
  e(D)\le\frac57n^2.
\end{equation*}
For every $W\subseteq V(D)$, the corresponding bound $e(D[W])\le\frac57|W|^2$ also holds.
\end{lemma}

\begin{proof}
The empty case is immediate, so assume $n\ge1$. First consider any nonempty $N$-vertex blow-up $G$ satisfying the assumptions of this section. Repeatedly delete a vertex of low total degree supplied by Lemma~\ref{lem:low-total-degree}. Every induced subdigraph, after empty classes are removed, satisfies the same assumptions, so the procedure continues to the empty graph. When the current graph has $m$ vertices, the number of arcs removed is the total degree of the deleted vertex and is strictly less than $10m/7$. Each arc is counted exactly once, when its first endpoint is deleted. Therefore
\begin{equation*}
  e(G)<\sum_{m=1}^N\frac{10}{7}m
  =\frac57N(N+1).
\end{equation*}

Return to $D$ and take $t=5n$. Replace every actual vertex by $t$ pairwise nonadjacent clones, obtaining $D^{(t)}$. It has the same quotient and
\[
  |V(D^{(t)})|=tn,
  \qquad
  e(D^{(t)})=t^2e(D).
\]
Applying the preceding estimate to $D^{(t)}$ and dividing by $t^2$ gives
\[
  e(D)<\frac57n^2+\frac{5n}{7t}
  =\frac57n^2+\frac17.
\]
Hence $7e(D)<5n^2+1$. Since $7e(D)$ and $5n^2$ are integers, $7e(D)\le5n^2$, proving the stated bound. The argument applies verbatim to every induced subdigraph.
\end{proof}

We will use the following weak-majorization lemma.

\begin{lemma}[Ai et al.~\cite{AiEtAl}]
	\label{lem:weak-majorization}
	Let
	\[
	a_1\ge a_2\ge\cdots\ge a_n\ge0
	\qquad\text{and}\qquad
	b_1\ge b_2\ge\cdots\ge b_n\ge0
	\]
	be two nonincreasing sequences of integers bounded by $M$. Suppose that
	\[
	\sum_{i=1}^j a_i
	\le
	\sum_{i=1}^j b_i
	\qquad\text{for every }1\le j\le n.
	\]
	Then, for every nondecreasing convex function
	$\varphi\colon[0,M]\to\mathbb R$,
	\[
	\sum_{i=1}^n\varphi(a_i)
	\le
	\sum_{i=1}^n\varphi(b_i).
	\]
\end{lemma}

\begin{proposition}\label{prop:no-core-deficit}
Let $n\ge1$, and let $D$ be an $n$-vertex independent blow-up of a $\Reg$-free semicomplete quotient containing no bidirected $K_4$. Then
\begin{equation}\label{eq:79-147-bound}
  \Phi(D)\le\frac{79}{147}n^3-\frac{4}{147}n.
\end{equation}
In particular, for every $n\ge5$,
\begin{equation*}
  \Phi(D)<B_4(n).
\end{equation*}
\end{proposition}

\begin{proof}
Order the out-degrees as $d_1\ge\cdots\ge d_n$, and let $W_j$ be the set of the first $j$ vertices. By Lemma~\ref{lem:finite-blowup-arc},
\begin{align}
  \sum_{i=1}^j d_i
  &=e(W_j)+e(W_j,V(D)\setminus W_j)\notag\\
  &\le\frac57j^2+j(n-j)
   =nj-\frac27j^2.
  \label{eq:degree-prefix-reg5}
\end{align}
The cross term counts only arcs directed from $W_j$ to its complement and is at most $j(n-j)$.

Set
\[
  z_i=n-\frac27(2i-1)
  \qquad(1\le i\le n).
\]
The first $j$ values of $z_i$ sum to $nj-2j^2/7$. Thus
\eqref{eq:degree-prefix-reg5} implies
\[
\sum_{i=1}^j 7d_i
\le
\sum_{i=1}^j 7z_i
\qquad(1\le j\le n).
\]
Both $(7d_i)$ and $(7z_i)$ are nonincreasing sequences of nonnegative
integers bounded by $7n$. Applying
Lemma~\ref{lem:weak-majorization} to these two sequences with
$\varphi(x)=x^2$, and then dividing by $49$, gives
\begin{align*}
  \Phi(D)
  &\le\sum_{i=1}^n\left[n-\frac27(2i-1)\right]^2\\
  &=n^3-\frac47n^3+\frac4{49}\frac{n(4n^2-1)}3
   =\frac{79}{147}n^3-\frac4{147}n.
\end{align*}
This proves \eqref{eq:79-147-bound}.

To compare with the balanced four-partite benchmark, write $n=4q+\rho$.
A direct expansion of the definition gives
\[
  B_4(n)
  =\frac{9n^3-\rho(4-\rho)(5n+2\rho-4)}{16}.
\]
Thus the correction subtracted from $9n^3/16$ is zero when $\rho=0$ and,
for $\rho=1,2,3$, is respectively
\[
  \frac{3(5n-2)}{16},
  \qquad
  \frac{5n}{4},
  \qquad
  \frac{3(5n+2)}{16},
\]
each at most $5n/4$. Hence
\[
  B_4(n)\ge\frac9{16}n^3-\frac54n.
\]
For $n\ge7$, combining this with \eqref{eq:79-147-bound} gives
\[
  B_4(n)-\Phi(D)
  \ge\frac{n}{2352}(59n^2-2876)>0,
\]
because $59\cdot7^2-2876=15$.

It remains to handle $n=5,6$. It is easy to know that $B_4(5)=66$ and $B_4(6)=114$. Since the left-hand side of \eqref{eq:degree-prefix-reg5} is integral, its right-hand side may be rounded down. For each $n\in\{5,6\}$, set
\[
  s_0=0,
  \qquad
  s_j=\left\lfloor nj-\frac27j^2\right\rfloor \quad (1\le j\le n),
  \qquad
  \widetilde z_i=s_i-s_{i-1}.
\]
The resulting prefix bounds, their difference sequences, and the corresponding square sums are
\[
\begin{array}{c|c|c|c|c}
 n & \text{prefix bounds} & (\widetilde z_i) & \sum \widetilde z_i^2 & B_4(n)\\
\hline
5&(4,8,12,15,17)&(4,4,4,3,2)&61&66\\
6&(5,10,15,19,22,25)&(5,5,5,4,3,3)&109&114.
\end{array}
\]
The two displayed difference sequences are nonincreasing nonnegative integer sequences, and \eqref{eq:degree-prefix-reg5} gives $\sum_{i=1}^j d_i\le s_j=\sum_{i=1}^j\widetilde z_i$ for every $j$. Applying Lemma~\ref{lem:weak-majorization} directly gives $\Phi(D)\le61<66$ and $\Phi(D)\le109<114$, respectively.
\end{proof}

Proposition~\ref{prop:no-core-deficit} rules out extremality when the symmetrized quotient has no bidirected $K_4$. We now treat the remaining case in next section.

\section{Proof of Theorem~\ref{thm:reg5-main}}\label{sec:four-deletion}

The main tool is a deletion estimate whose equality case identifies the extremal graph. We first state it and derive Theorem~\ref{thm:reg5-main}, and then prove it thereafter.
%Sections~\ref{subsec:proof-four-deletion}.
%\subsection{The deletion statement and the main theorem}\label{subsec:deletion-statement}

\begin{theorem}\label{thm:four-vertex-deletion}
Let $D$ be an $n$-vertex independent blow-up whose semicomplete quotient $Q$ is $\Reg$-free and contains a $\dd\!_4$. Let $V_1,V_2,V_3,V_4$ be the four corresponding independent classes, choose $c_i\in V_i$, and put
\[
  C=\{c_1,c_2,c_3,c_4\},
  \qquad
  t=|V(D)\setminus(V_1\cup V_2\cup V_3\cup V_4)|.
\]
Then
\begin{equation*}
  \Phi(D)-\Phi(D-C)
  \le B_4(n)-B_4(n-4)-t.
\end{equation*}
In particular,
\begin{equation}\label{eq:four-delete-plain}
  \Phi(D)-\Phi(D-C)
  \le B_4(n)-B_4(n-4).
\end{equation}
Equality in \eqref{eq:four-delete-plain} holds if and only if $D\cong \dT\!_4(n)$.
\end{theorem}

\begin{proof}[\bfseries{Proof of Theorem~\ref{thm:reg5-main}}]
The digraph $\dT\!_4(n)$ is $\Reg$-free, and its degree-square sum
equals the value stated in the theorem. It remains to prove the upper
bound and uniqueness.

For $n\le4$, every digraph is $\Reg$-free. Obviously, the unique extremal digraph is $\dd\!_n=\dT\!_4(n)$.

Assume $n\ge5$ and argue by strong induction on $n$. Let $D$ be an $n$-vertex $\Reg$-free digraph maximizing $\Phi$. By Lemma~\ref{lem:reg5-symmetrization}, there is an independent blow-up $\widehat D$ with semicomplete quotient $Q$ such that every labelled out-degree is preserved. In particular,
\begin{equation*}
  \Phi(D)=\Phi(\widehat D).
\end{equation*}

If $Q$ contained no $\dd\!_4$, Proposition~\ref{prop:no-core-deficit} would give $\Phi(\widehat D)<B_4(n)$, a contradiction. Hence $Q$ contains a $\dd\!_4$. Choose one vertex from each of its four blow-up classes to form $C$. By Theorem~\ref{thm:four-vertex-deletion} and the induction hypothesis,
\begin{align*}
  \Phi(D)=\Phi(\widehat D)
  &\le\Phi(\widehat D-C)+B_4(n)-B_4(n-4)\\
  &\le B_4(n-4)+B_4(n)-B_4(n-4)\\
  &=B_4(n).
\end{align*}
Here $\widehat D-C$ is merely an $(n-4)$-vertex $\Reg$-free digraph. Since $\dT\!_4(n)$ provides the lower bound, this establishes the exact value.

Now suppose that the original digraph $D$ is extremal. Every inequality above is then an equality, and in particular equality holds in \eqref{eq:four-delete-plain} for $\widehat D$ and $C$. The equality statement in Theorem~\ref{thm:four-vertex-deletion} therefore gives
\begin{equation*}
  \widehat D\cong \dT\!_4(n).
\end{equation*}
It remains to recover $D$. The nonadjacency classes used to construct $\widehat D$ are independent in $D$, while any two distinct classes are adjacent. Since $\widehat D\cong \dT\!_4(n)$, these classes are precisely the four balanced independent parts of $\dT\!_4(n)$. Fix $x\in V_i$. We have
\[
  d_D^+(x)=d_{\widehat D}^+(x)=n-|V_i|.
\]
Since $V_i$ is already independent in $D$, all out-neighbors of $x$ lie outside $V_i$, and there are exactly $n-|V_i|$ such vertices. Thus $x$ sends an arc to every vertex outside its class. Applying this to every vertex shows that both directions are present between every two different parts. Therefore $D\cong \dT\!_4(n)$. The converse was verified by the lower-bound construction.
\end{proof}

\vskip2mm
Now, we are left to complete the proof of  Theorem~\ref{thm:four-vertex-deletion}.

Let $D$,
$V_1,\ldots,V_4$, and $C=\{c_1,c_2,c_3,c_4\}$ be as in Theorem \ref{thm:four-vertex-deletion}. Write
\begin{equation*}
  % \qquad
  Y=V(D)\setminus C,
  \qquad
  b=|Y|=n-4,
\end{equation*}
and
\begin{equation*}
  t=|V(D)\setminus(V_1\cup V_2\cup V_3\cup V_4)|.
\end{equation*}
The four core vertices form pairwise digons. For $x\in Y$, define
\begin{equation*}
  O_x=N_D^+(x)\cap C,
  \qquad
  I_x=N_D^-(x)\cap C,
  \qquad
  h_x=|O_x|,
  \qquad
  k_x=|I_x|.
\end{equation*}
If $x\in V_i\setminus\{c_i\}$, then
\begin{equation*}
  O_x=I_x=C\setminus\{c_i\},
  \qquad
  h_x=k_x=3.
\end{equation*}

%The rest of this section is devoted to the proof of 

%\subsection{Proof of Theorem~\ref{thm:four-vertex-deletion}}\label{subsec:proof-four-deletion}

%
\vskip 2mm
We first establish four auxiliary lemmas and then combine them to prove
Theorem~\ref{thm:four-vertex-deletion}.
 %Throughout this subsection, 

\begin{lemma}\label{lem:core-types}
If $x$ lies outside the four core classes, then
\begin{equation*}
  h_x\le1
  \quad\text{or}\quad
  k_x\le1,
\end{equation*}
and the two inequalities cannot both hold.
\end{lemma}

\begin{proof}
The quotient is semicomplete and $x$ belongs to none of the four core classes, so $O_x\cup I_x=C$. Suppose that $h_x,k_x\ge2$. The set $C\setminus I_x$ has at most two vertices and is contained in $O_x$. Choose a two-element set $O\subseteq O_x$ containing $C\setminus I_x$. Then $C\setminus O\subseteq I_x$. Select the four arcs $x\to O$ and $C\setminus O\to x$, so that $x$ has exactly two out-neighbors and two in-neighbors in the core. Since every core pair is a digon, the remaining six arcs can be selected to complete a copy of $\Reg$, a contradiction.

Also, $O_x\cup I_x=C$ implies $h_x+k_x\ge4$, so both $h_x$ and $k_x$ cannot be at most one.
\end{proof}

Using Lemma~\ref{lem:core-types}, partition $Y$ into two blocks:
\begin{equation*}
  L=\{x\in Y:h_x\le1\},
  \qquad
  H=Y\setminus L,
  \qquad
  \ell=|L|,
  \qquad
  a=|H|.
\end{equation*}
Then
\begin{equation*}
  b=a+\ell,
  \qquad
  0\le\ell\le t.
\end{equation*}
The block $H$ contains all remaining vertices in the four core classes, as well as every external vertex with $k_x\le1$. Lemma~\ref{lem:core-types} and the description of vertices in the core classes show that every $x\in H$ satisfies
\begin{equation*}
  3\le h_x\le4,
  \qquad
  k_x\le3,
  \qquad
  h_x+k_x\le6.
\end{equation*}
If $x\in H$ lies outside the core classes, then in fact $k_x\le1$.

To record the quadratic terms arising when $C$ is deleted, for $X\subseteq Y$ set
\begin{equation*}
  z_i(X)=|N_D^+(c_i)\cap X|,
  \qquad
  P(X)=\sum_{i=1}^4z_i(X)^2+2\sum_{x\in X}h_xd_X^+(x).
\end{equation*}

The next calculation identifies the balanced four-partite benchmark and
will also be used to determine the equality case.

\begin{lemma}\label{lem:four-partite-increment}
Let $G$ be a complete bidirected four-partite graph with part sizes
$b_1,\ldots,b_4\ge0$ and total order $b$, where empty parts are allowed.
Let $G'$ be obtained by adding one vertex to each of the four parts. Then
\begin{equation}\label{eq:four-partite-increment}
  \Phi(G')-\Phi(G)=36+27b+8b^2-5\sum_{i=1}^4b_i^2.
\end{equation}
In particular,
\begin{equation}\label{eq:balanced-four-increment}
  B_4(b+4)-B_4(b)=36+27b+3b^2+5e_4(b).
\end{equation}
\end{lemma}

\begin{proof}
Put $d_i=b-b_i$. Every old out-degree increases by three, and the new
vertex in part $i$ has out-degree $d_i+3$. Hence
\begin{align*}
  \Phi(G')-\Phi(G)
  &=\sum_{i=1}^4(d_i+3)^2
    +\sum_{i=1}^4b_i\bigl((d_i+3)^2-d_i^2\bigr)\\*
  &=36+27b+8b^2-5\sum_{i=1}^4b_i^2,
\end{align*}
which proves \eqref{eq:four-partite-increment}. For the balanced
four-partite graph, $\sum_i b_i^2=p_4(b)=b^2-e_4(b)$, and adding one vertex
to each part preserves balance. Substitution gives
\eqref{eq:balanced-four-increment}.
\end{proof}

We now estimate the quadratic term $P(Y)$, starting with its restriction
to $H$.

\begin{lemma}\label{lem:P-H}
For the block $H$ defined above,
\begin{equation}\label{eq:P-H-bound}
  P(H)\le3a^2+5e(H).
\end{equation}
\end{lemma}

\begin{proof}
For distinct $x,y\in Y$, write
\[
  m_{xy}=\ind_{xy\in A(D)}+\ind_{yx\in A(D)}\in\{0,1,2\}.
\]
We first expand $P(X)$. Since
\[
  z_i(X)=\sum_{x\in X}\ind_{c_i\in I_x},
\]
we have
\[
  z_i(X)^2
  =\sum_{x\in X}\ind_{c_i\in I_x}
   +2\sum_{\{x,y\}\subseteq X}
      \ind_{c_i\in I_x}\ind_{c_i\in I_y}.
\]
Summing over $i$ yields
\[
  \sum_{i=1}^4z_i(X)^2
  =\sum_{x\in X}k_x
   +2\sum_{\{x,y\}\subseteq X}|I_x\cap I_y|.
\]
Also,
\[
  \sum_{x\in X}h_xd_X^+(x)
  =\sum_{xy\in A(D[X])}h_x.
\]
Therefore, for every $X\subseteq Y$,
\begin{equation}\label{eq:P-pair-expansion}
  P(X)
  =\sum_{x\in X}k_x
   +2\sum_{\{x,y\}\subseteq X}|I_x\cap I_y|
   +2\sum_{xy\in A(D[X])}h_x.
\end{equation}

The right-hand side of \eqref{eq:P-H-bound} can be written as
\[
  3a^2+5e(H)=3a+6\binom a2+5e(H).
\]
This suggests a budget of $3$ for every vertex, $6$ for every unordered pair, and an additional $5$ for every arc between the pair. Thus, for an ordinary pair $\{x,y\}$, we seek the bound
\begin{equation*}
  \text{pair contribution}\le6+5m_{xy}.
\end{equation*}
Since $k_x\le3$ on $H$, the single-vertex contribution already fits the budget of three. One exceptional kind of digon pair must be handled collectively.

For $i=1,2,3,4$, define
\begin{equation*}
  T_i=\{x\in H:h_x=4,\ I_x=\{c_i\}\},
  \qquad
  t_i=|T_i|.
\end{equation*}
These sets are pairwise disjoint. Their collective treatment is essential: if two vertices in the same $T_i$ form a digon, their pair contribution is $2+8+8=18$, exceeding the proposed budget $6+5m_{xy}=16$. We control this excess using an arc bound for the entire block $T_i$.

Every vertex of $T_i$ forms a digon with $c_i$. If $D[T_i]$ contained a copy of $R_4$, the digons between $c_i$ and the four vertices of that copy would extend it to a copy of $\Reg$, a contradiction. Thus $D[T_i]$ is $R_4$-free, and Theorem~\ref{thm:Brown-Harary} gives
\begin{equation*}
  e(T_i)\le2\ex(t_i,K_4)\le\frac23t_i^2.
\end{equation*}
For $x\in T_i$, we have $k_x=1$ and $h_x=4$, and every pair $x,y\in T_i$ satisfies $I_x=I_y=\{c_i\}$. Hence \eqref{eq:P-pair-expansion} gives
\[
  \sum_{x\in T_i}k_x
  +2\sum_{\{x,y\}\subseteq T_i}|I_x\cap I_y|
  =t_i+2\binom{t_i}{2}=t_i^2.
\]
Every internal arc of $T_i$ contributes $2h_x=8$ to the last term, and therefore
\begin{equation*}
  P(T_i)=t_i^2+8e(T_i)
  \le3t_i^2+5e(T_i),
\end{equation*}
where the last inequality follows from the preceding bound $e(T_i)\le\frac23t_i^2$.

Now consider an unordered pair $\{x,y\}\subseteq H$ that is not contained in one common $T_i$. By \eqref{eq:P-pair-expansion}, its pair contribution is
\begin{equation*}
  2|I_x\cap I_y|
  +2h_x\ind_{xy\in A(D)}
  +2h_y\ind_{yx\in A(D)}.
\end{equation*}
If $m_{xy}=0$, then this expression is $2|I_x\cap I_y|$. Since $k_x,k_y\le3$, it is at most six, as required.

If $m_{xy}=1$, the two vertices cannot both lie in the four core classes: vertices in the same core class are nonadjacent, while vertices in distinct core classes form a digon. Hence at least one of $x,y$ is external. Because it lies in $H$, its $k$-value is at most one, so $|I_x\cap I_y|\le1$. The tail of the unique arc has $h$-value at most four. Thus the total is at most
\[
  2+2\cdot4=10\le11=6+5m_{xy}.
\]

If $m_{xy}=2$, then $x\leftrightarrow y$, and the pair contribution is
\[
  2\bigl(|I_x\cap I_y|+h_x+h_y\bigr).
\]
If both vertices lie in core classes, then they lie in different classes, $|I_x\cap I_y|=2$, and $h_x=h_y=3$, so the contribution is sixteen. Otherwise at least one vertex is external, and hence $|I_x\cap I_y|\le1$. Since $h_x,h_y\le4$, the quantity inside parentheses can exceed eight only in the integer case
\[
  |I_x\cap I_y|=1,
  \qquad
  h_x=h_y=4.
\]
Neither vertex can then lie in a core class. Both are external vertices of $H$, so $k_x,k_y\le1$; the common in-neighbor forces $I_x=I_y=\{c_i\}$ for some $i$, placing both vertices in $T_i$, contrary to the current assumption. Thus the pair contribution is at most
\[
  16=6+5\cdot2=6+5m_{xy}.
\]

Finally, sum all budgets. In the preceding estimate for $P(T_i)$, the baseline term
\[
  3t_i^2=3t_i+6\binom{t_i}{2}
\]
is exactly the allocation of three to each vertex and six to each internal unordered pair. All other vertices and pairs use the same baseline. The total constant budget is 
\[
  3a+6\binom a2=3a^2,
\]
and every arc in $H$ receives coefficient five. This proves the lemma.
\end{proof}

The next lemma adds the contribution of $L$ and the cross terms between
$H$ and $L$.

\begin{lemma}\label{lem:P-Y}
With the notation above,
\begin{equation}\label{eq:P-target}
  P(Y)\le3b^2+5e_4(b)+\ell.
\end{equation}
\end{lemma}

\begin{proof}
Since $z_i(L)\le\ell$ and $h_x\le1$ for $x\in L$,
\begin{equation*}
  P(L)\le4\ell^2+2e(L).
\end{equation*}
The cross-block part of $P$ is
\begin{align*}
  R&:=P(Y)-P(H)-P(L)\\
   &=2\sum_{i=1}^4z_i(H)z_i(L)
     +2\sum_{x\in H}h_xd_L^+(x)
     +2\sum_{x\in L}h_xd_H^+(x).
\end{align*}
The three sums satisfy
\begin{align*}
  \sum_{i=1}^4z_i(H)z_i(L)
  &\le\ell\sum_{i=1}^4z_i(H)
   =\ell\sum_{x\in H}k_x,\\
  \sum_{x\in H}h_xd_L^+(x)
  &=e(H,L)+\sum_{x\in H}(h_x-1)d_L^+(x)\\
  &\le e(H,L)+\ell\sum_{x\in H}(h_x-1),\\
  \sum_{x\in L}h_xd_H^+(x)
  &\le e(L,H).
\end{align*}
Using $h_x+k_x-1\le5$ for every $x\in H$, we obtain
\begin{align*}
  R
  &\le2\bigl(e(H,L)+e(L,H)\bigr)
      +2\ell\sum_{x\in H}(h_x+k_x-1)\\
  &\le2\bigl(e(H,L)+e(L,H)\bigr)+10a\ell.
\end{align*}
Combining Lemma~\ref{lem:P-H} with the two preceding bounds gives
\begin{equation*}
  P(Y)
  \le2e(Y)+3e(H)+3a^2+10a\ell+4\ell^2.
\end{equation*}
Both $D[Y]$ and $D[H]$ are $\Reg$-free. By
Theorem~\ref{thm:Brown-Harary} and Tur\'an's theorem, every $m$-vertex
$\Reg$-free digraph $G$ satisfies
\[
  e(G)\le 2\ex(m,K_5)=m^2-p_4(m)=e_4(m).
\]
Therefore $e(Y)\le e_4(b)$ and $e(H)\le e_4(a)$, so
\begin{equation*}
  P(Y)
  \le2e_4(b)+3e_4(a)+3a^2+10a\ell+4\ell^2.
\end{equation*}

To obtain \eqref{eq:P-target} from this bound, we compare $e_4(a+\ell)$ with $e_4(a)$. Adding one vertex to a smallest part of $\dT\!_4(m)$ gives $\dT\!_4(m+1)$, and therefore
\[
  e_4(m+1)-e_4(m)
  =2\left(m-\left\lfloor\frac m4\right\rfloor\right)
  \ge\frac32m.
\]
Summing from $m=a$ to $a+\ell-1$ and multiplying by three yields
\begin{equation*}
  3\bigl(e_4(a+\ell)-e_4(a)\bigr)
  \ge\frac92a\ell+\frac94\ell(\ell-1)
  \ge4a\ell+\ell(\ell-1).
\end{equation*}
Since $b=a+\ell$, the difference between the target in \eqref{eq:P-target} and the preceding right-hand side is
\begin{align*}
  &3b^2+5e_4(b)+\ell
   -\bigl(2e_4(b)+3e_4(a)+3a^2+10a\ell+4\ell^2\bigr)\\
  &\qquad=3\bigl(e_4(a+\ell)-e_4(a)\bigr)
           -4a\ell-\ell(\ell-1)
  \ge0,
\end{align*}
where the final inequality follows from the preceding estimate. Thus \eqref{eq:P-target} holds.
\end{proof}

We can now prove the deletion estimate and identify its equality case.

\begin{proof}[\bfseries{Proof of Theorem~\ref{thm:four-vertex-deletion}}]
The core vertex $c_i$ has out-degree $3+z_i(Y)$, while a vertex $x\in Y$
has out-degree $d_Y^+(x)+h_x$ before deleting the core and $d_Y^+(x)$
afterwards. Set $\Delta_C(D)=\Phi(D)-\Phi(D-C)$. Since
$\sum_i z_i(Y)=\sum_{x\in Y}k_x$, direct expansion gives
\begin{align*}
  \Delta_C(D)
  &=\sum_{i=1}^4(3+z_i(Y))^2
    +\sum_{x\in Y}\bigl((d_Y^+(x)+h_x)^2-d_Y^+(x)^2\bigr)\\
  &=36+\sum_{x\in Y}(h_x^2+6k_x)+P(Y).
\end{align*}
For each of the $b-t$ remaining vertices in the core classes,
$h_x=k_x=3$, so $h_x^2+6k_x=27$. For every external vertex,
Lemma~\ref{lem:core-types} gives
\[
  h_x^2+6k_x
  \le
  \begin{cases}
    1+6\cdot4=25,& h_x\le1,\\
    4^2+6\cdot1=22,& k_x\le1.
  \end{cases}
\]
Thus each external vertex contributes at most $25$, and
\begin{equation}\label{eq:deletion-linear-margin}
  \Delta_C(D)\le36+27b-2t+P(Y).
\end{equation}
Combining Lemmas~\ref{lem:four-partite-increment} and~\ref{lem:P-Y} with
$\ell\le t$ and $b=n-4$, we obtain
\begin{align*}
  \Delta_C(D)
  &\le36+27b-2t+3b^2+5e_4(b)+\ell\\
  &=B_4(n)-B_4(n-4)-2t+\ell\\
  &\le B_4(n)-B_4(n-4)-t.
\end{align*}
This proves the quantitative estimate and hence
\eqref{eq:four-delete-plain}.

Suppose equality holds in \eqref{eq:four-delete-plain}. The quantitative
estimate forces $t=0$, so $D$ is a complete bidirected four-partite graph.
Write its part sizes as
\[
  r_1+1,\ldots,r_4+1,
  \qquad r_i\ge0,
  \qquad \sum_{i=1}^4r_i=b.
\]
Deleting $C$ leaves parts of sizes $r_1,\ldots,r_4$. Applying
Lemma~\ref{lem:four-partite-increment} to this graph and to the balanced
four-partite graph gives the exact gap
\[
  \bigl[B_4(n)-B_4(n-4)\bigr]-\Delta_C(D)
  =5\left(\sum_{i=1}^4r_i^2-p_4(b)\right).
\]
Since $\sum_i r_i=b$, transferring one unit from a term that exceeds
another by at least two strictly decreases $\sum_i r_i^2$. Thus
$\sum_i r_i^2\ge p_4(b)$, with equality if and only if $r_1,\ldots,r_4$
differ pairwise by at most one. Consequently, equality in
\eqref{eq:four-delete-plain} forces the four part sizes $r_i+1$ to be
balanced, so $D\cong \dT\!_4(n)$.

Conversely, deleting one vertex from each part of $\dT\!_4(n)$ produces
$\dT\!_4(n-4)$, so equality holds in \eqref{eq:four-delete-plain}.
\end{proof}

\section{Concluding remarks}

We have determined the degree-square Tur\'an number for two classes of
tournaments. For transitive tournaments, our pointwise domination
theorem extends the exact result for $\TT_4$ to all $\TT_r$ and, more
generally, reduces every nondecreasing out-degree functional to an
integer-partition problem. For the regular tournament $\Reg$ on five
vertices, we prove the conjectured exact value and show that the balanced
complete bidirected four-partite digraph is the unique extremal
construction. In particular, these results answer the problem of Ai et al. \cite{AiEtAl} for all transitive tournaments and for the regular tournament on five vertices.

It remains natural to ask whether the same phenomenon occurs for other
forbidden tournaments. That is, for which tournaments $F$ is
$\ex_2^+(n,F)$ attained uniquely by a balanced complete bidirected
multipartite digraph? A further problem is stability: if an $F$-free
digraph has degree-square sum close to the maximum, must it be close,
after changing few arcs, to the corresponding extremal construction?

\section*{Acknowledgments}
This research was supported by the National Key R\&D Program of China under grant number 2024YFA1013900, the National Natural Science Foundation of China under grant number 12471327.

\end{document}